\documentclass[15pt]{article}

\usepackage{amsmath,amsfonts,amssymb,amsthm,amscd}

\newcounter{stepctr}
{\end{list}}

\newtheorem{thm}{Theorem}[section]

\newtheorem{cor}[thm]{Corollary}

\theoremstyle{definition}
\newtheorem{dfn}[thm]{Definition}
\newtheorem{ex}[thm]{Example}

\newtheorem{rema}[thm]{Remark}
\newtheorem{lem}[thm]{Lemma}
\newtheorem{prob*}{Open problem}
\newcommand{\demo}{\begin{proof}}

\newcommand{\R}{\ensuremath{\mathcal{R}}}

\newcommand{\N}{\mathbb{N}}

\def\ll^2{{\mathcal L}(\ell^2(\N))}

\def\f^0x{{\mathcal F^0}(X) }

\title
{\bf  Extended   Rako\v{c}evi\'{c}'s property}

\author{Kaoutar Ben Ouidren, Hassan  Zariouh }

\date{}
\begin{document}
\maketitle \thispagestyle{empty}
\begin{abstract}\noindent\baselineskip=10pt
The purpose of this paper is to introduce and study  new extension of   Rako\v{c}evi\'{c}'s property $(w)$ and   property $(b)$ introduced by Berkani--Zariouh in \cite{berkani-zariouh1}, in connection with other Weyl type theorems and recent properties. We prove in particular,  the two following results: \\
1. A bounded linear  operator  $T$ satisfies property $(w_{\pi_{00}})$ if and only if $T$ satisfies property  $(w)$ and  $\sigma_{uf}(T)=\sigma_{uw}(T).$\\
2. $T$ satisfies property $(gw_{\pi_{00}})$ if and only if $T$ satisfies property  $(w_{\pi_{00}})$ and  $\pi_{0}(T)=p_{0}^a(T).$ Classes of operators are considered as illustrating examples.

\end{abstract}

 \baselineskip=15pt
 \footnotetext{\small \noindent  2010 AMS subject
classification: Primary 47A53, 47A10, 47A11, 47A55  \\
\noindent Keywords: property $(gbz_{1}),$ property $(gw_{\pi_{00}}),$  upper semi-B-Fredholm spectrum }
 \baselineskip=15pt

\section{Introduction and basic definitions }

This paper is a continuation of \cite{kaoutar-zariouh}, where we introduced and
studied  the new approach to a-Weyl's theorem. Our purpose is
to investigate  new  spectral properties named  $(bz_{1}),$ $(gbz_{1}),$ $(w_{\pi_{00}})$ and
$(gw_{\pi_{00}})$ (see Definitions \ref{defn1} and \ref{defn2}) for bounded
linear operators as new versions of properties $(bz),$  $(w_{\pi_{00}^{a}})$ and $(gw_{\pi_{00}^{a}}),$ which have been introduced and studied in \cite{kaoutar-zariouh}. The
results obtained are summarized in the diagram presented at the end
of this paper. In order to simplify, we use the same symbols and notations used in \cite{kaoutar-zariouh}. For more details on  several  classes and spectra originating from  Fredholm theory or  B-Fredholm theory, we refer the reader to \cite{kaoutar-zariouh, berkani-zariouh1}. The following  list, summarizes  the same   notations and symbols used in \cite{kaoutar-zariouh}, which will be
needed in the sequel.\\

 \begin{tabular}{l|l}
   $\sigma_{b}(T)$:  Browder spectrum of $T$ &  $\Delta^g(T):=\sigma(T)\setminus\sigma_{bw}(T)$\\
   $\sigma_{uf}(T)$:  upper semi-Fredholm spectrum of $T$ & $\Delta_{uf}^g(T):=\sigma_{a}(T)\setminus\sigma_{ubf}(T)$\\
   $\sigma_{ubf}(T)$:  upper semi-B-Fredholm spectrum of $T$ & $\Delta_{uf}(T):=\sigma_{a}(T)\setminus\sigma_{uf}(T)$  \\
   $\sigma_{ub}(T)$:  upper semi-Browder spectrum of $T$ & $\Delta_{a}^g(T):=\sigma_{a}(T)\setminus\sigma_{ubw}(T)$\\
   $\sigma_{w}(T)$: Weyl spectrum of $T$ & $\Delta(T):=\sigma(T)\setminus\sigma_{w}(T)$\\
   $\sigma_{bw}(T)$: B-Weyl spectrum of $T$ & $p_{00}(T)$: poles of $T$ of finite rank\\
   $\sigma_{uw}(T)$: upper semi-Weyl spectrum of $T$ & $p_{0}^a(T)$: left  poles of $T$\\
   $\sigma_{ubw}(T)$: upper semi-B-Weyl spectrum of $T$ & $p_{00}^a(T)$: left  poles of $T$ of finite rank\\
   $\sigma_{ld}(T)$: left Drazin spectrum of $T$ & $p_{0}(T)$: poles of $T$\\
   $ \sigma_{d}(T)$:  Drazin spectrum of $T,$& $\pi_{0}^a(T):=\mbox{iso}\,\sigma_a(T)\cap\sigma_{p}(T)$\\
    $\sigma_{p}(T)$:  eigenvalues of $T$ & $\pi_{0}(T):=\mbox{iso}\,\sigma(T)\cap\sigma_{p}(T)$\\
$\sigma_{p}^f(T)$: eigenvalues of $T$ of finite multiplicity & $\pi_{00}(T):=\mbox{iso}\,\sigma(T)\cap\sigma_{p}^f(T)$\\
$\Delta_{a}(T):=\sigma_{a}(T)\setminus\sigma_{uw}(T)$ & $\pi_{00}^a(T):=\mbox{iso}\,\sigma_a(T)\cap\sigma_{p}^f(T)$\\
\end{tabular}\\

\begin{dfn}\cite{ amouch-berkani, kaoutar-zariouh, berkani-zariouh1, rakocevic} Let $T\in L(X).$ We say that\\
(i) $T$ satisfies property  $(w)$ if $\Delta_{a}(T)=\pi_{00}(T).$  \\
(ii) $T$ satisfies property  $(gw)$ if $\Delta_{a}^g(T)=\pi_{0}(T).$\\
(iii) $T$ satisfies property  $(b)$ if $\Delta_{a}(T)=p_{00}(T).$\\
(iv) $T$ satisfies property  $(gb)$ if $\Delta_{a}^g(T)=p_{0}(T).$\\
(v) $T$ satisfies property  $(bz)$ if $\Delta_{uf}(T)=p_{00}^a(T).$\\
(vi) $T$ satisfies property  $(gbz)$ if $\Delta_{uf}^g(T)=p_{0}^a(T).$
\end{dfn}

To give the reader a good overview of the subject, we
 summarize in the following  theorem  some known results which are useful in this paper.

\begin{thm}{\rm\cite{ amouch-berkani, kaoutar-zariouh, berkani-zariouh1}} \label{thm0} The following statements hold  for every $T\in L(X).$\\
(i) Property $(gb)$ holds for $T$ if and only if property $(b)$ holds for $T$ and $p_{0}(T)=p_{0}^a(T).$\\
(ii) Property $(gw)$ holds for $T$ if and only if property $(w)$ holds for $T$ and $p_{0}^a(T)=\pi_{0}(T).$\\
(iii) $\sigma_{ubf}(T)=\sigma_{ubw}(T)$ if and only if $\sigma_{uf}(T)=\sigma_{uw}(T).$\\
(iv) Property $(gbz)$ holds for $T$ if and only if property $(bz)$ holds for $T.$
\end{thm}

\section{Remarks and improvement}

The following property has relevant role in local
spectral theory: a bounded linear
 operator $T\in L(X)$ is said to have the {\it single-valued
 extension property} (SVEP for short) at $\lambda\in\mathbb{C}$ if
  for every open neighborhood $U_\lambda$ of $\lambda,$ the  function $f\equiv 0$ is the only
 analytic solution of the equation
 $(T-\mu I)f(\mu)=0\quad\forall\mu\in U_\lambda.$ We denote by
  ${\mathcal S}(T)=\{\lambda\in\mathbb{C}: T\mbox{  does not have the SVEP at } \lambda\}$
     and we say that  $T$ has  SVEP
  if
$ {\mathcal S}(T)=\emptyset.$ We say that $T$ has the SVEP on $A\subset\mathbb{C},$ if $T$ has the SVEP
at every $\lambda\in A.$ (For more details about this property, we refer the reader to \cite{aiena1}).

From \cite[Theorem 2.7]{kaoutar-zariouh}, we have $T\in L(X)$ has the SVEP on $(\sigma_{uf}(T))^C$ if and only if $T$ has the SVEP on $(\sigma_{ubf}(T))^C,$ where $(\sigma_{uf}(T))^{C}$ and $(\sigma_{ubf}(T))^{C}$ means respectively, the complementary of upper semi-Fredholm spectrum  and the complementary of upper semi-B-Fredholm spectrum of an operator $T.$

\begin{lem}\label{lem1} For every $T\in L(X),$ we have the following equivalence $$\sigma_{uw}(T)=\sigma_{w}(T) \Longleftrightarrow \sigma_{ubw}(T)=\sigma_{bw}(T).$$
\end{lem}
\begin{proof}

 $\Longleftarrow$) Let $\lambda\not\in\sigma_{uw}(T)$ be arbitrary and without loss of generality we can assume that $\lambda=0.$ So $T$ is an upper semi-Weyl and in particular, it is an upper semi-B-Weyl. Since $\sigma_{ubw}(T)=\sigma_{bw}(T),$  it follows that $0\not\in\sigma_{w}(T)$ and hence $\sigma_{uw}(T)=\sigma_{w}(T).$\\
 $\Longrightarrow$) Suppose that $0\not\in\sigma_{ubw}(T).$  Then from the punctured
neighborhood theorem for semi-B-Fredholm operators \cite[Corollary 3.2]{berkani-sarih}, there exists $\epsilon>0$ such that $T-\mu I $ is an upper semi-Fredholm and $\mbox{ind}(T-\mu I)=\mbox{ind}(T)$ for every $0<|\mu|<\epsilon.$ This implies that  $\mu\not\in\sigma_{uw}(T)=\sigma_{w}(T)$  and so $\mbox{ind}(T-\mu I)=\mbox{ind}(T)=0.$ Hence
 $\sigma_{ubw}(T)=\sigma_{bw}(T).$
 \end{proof}

 Now, we give an improvement of  some results of \cite{berkani-sarih-zariouh}. But to give the reader a global and precise view, we recall here these results  in the same order of the paper \cite{berkani-sarih-zariouh}.\\

 \noindent {\bf Lemma I}: \cite[Lemma 2.1]{berkani-sarih-zariouh} Let $T\in L(X).$  If $T^*$ has the SVEP on $(\sigma_{uw}(T))^{C},$ then $\sigma(T)=\sigma_{a}(T),$ $\sigma_{uw}(T)=\sigma_{w}(T).$\\

\noindent  {\bf Theorem I}: \cite[Theorem 2.2]{berkani-sarih-zariouh} Let $T\in L(X).$ Then the following assertions are equivalent. \\
(a) $T^*$ has the SVEP on $(\sigma_{uw}(T))^{C};$ \\
(b)  $T$ satisfies property $(b)$ and $\sigma_{uw}(T)=\sigma_{w}(T).$\\

  \noindent   {\bf Lemma II}: \cite[Lemma 2.4]{berkani-sarih-zariouh}   Let $T\in L(X).$  If $T^*$ has the SVEP on $(\sigma_{ubw}(T))^{C},$ then $\sigma(T)=\sigma_{a}(T),$ $\sigma_{ubw}(T)=\sigma_{bw}(T)=\sigma_{d}(T).$\\

  \noindent    {\bf Theorem II}: \cite[Theorem 2.5]{berkani-sarih-zariouh} Let $T\in L(X).$ Then the following assertions are equivalent.\\
   (c) $T^*$ has the SVEP on $(\sigma_{ubw}(T))^{C};$ \\
   (d)  $T$ satisfies property $(gb)$ and $\sigma_{ubw}(T)=\sigma_{bw}(T).$\\

 Using Lemma \ref{lem1}, we  remark that  Lemma II is an immediate consequence [since $(\sigma_{uw}(T))^{C}\subset (\sigma_{ubw}(T))^{C}$] of   Lemma I and  can be combined in a single  as follows:
  \begin{lem}\label{lem2} Let $T\in L(X).$ If $T^*$ has the SVEP on $(\sigma_{uw}(T))^{C},$ then we have \\
    (i) $\sigma(T)=\sigma_{a}(T),$ $\sigma_{ub}(T)=\sigma_{uw}(T)=\sigma_{w}(T)=\sigma_{b}(T).$\\
    (ii)   $\sigma_{ld}(T)=\sigma_{ubw}(T)=\sigma_{bw}(T)=\sigma_{d}(T),$ $p_{0}(T)=p_{0}^a(T)$ and $p_{00}(T)=p_{00}^a(T).$
  \end{lem}

  \begin{proof}  (i) See the proof of \cite[Lemma 2.1]{berkani-sarih-zariouh}.\\
  (ii) From Lemma \ref{lem1}, we have  $\sigma_{ubw}(T)=\sigma_{bw}(T).$ It is well known that $\sigma_{w}(T)=\sigma_{b}(T)$ is equivalent to say that $\sigma_{bw}(T)=\sigma_{d}(T),$ and $\sigma_{uw}(T)=\sigma_{ub}(T)$ is equivalent to say that $\sigma_{ubw}(T)=\sigma_{ld}(T).$ Thus $\sigma_{ld}(T)=\sigma_{ubw}(T)=\sigma_{bw}(T)=\sigma_{d}(T).$ So $p_{0}(T)=\sigma(T)\setminus \sigma_{d}(T)= \sigma_{a}(T)\setminus \sigma_{ld}(T)=p_{0}^a(T),$ and this implies that $p_{00}(T)=p_{00}^a(T).$
  \end{proof}

     We prove in the following corollary, that the statements (a), (b), (c) and (d) of Theorem I and Theorem II are equivalent.

  \begin{cor} For every $T\in L(X),$ the following statement are equivalent.\\
  (i) $T^*$ has the SVEP on $(\sigma_{uw}(T))^{C};$\\
  (ii) $T$ satisfies property $(b)$ and  $\sigma_{uw}(T)=\sigma_{w}(T);$\\
  (iii)  $T^*$ has the SVEP on $(\sigma_{ubw}(T))^{C};$\\
 (iv)  $T$ satisfies property $(gb)$ and  $\sigma_{ubw}(T)=\sigma_{bw}(T).$
  \end{cor}
  \begin{proof}(i) $\Longleftrightarrow$ (ii) See \cite[Theorem 2.2]{berkani-sarih-zariouh}.\\
    (iii) $\Longleftrightarrow$ (iv) See \cite[Theorem 2.5]{berkani-sarih-zariouh}.\\
    (iii) $\Longrightarrow$ (i) Obvious. \\
    (i) $\Longrightarrow$ (iv) If  $T^*$ has the SVEP on $(\sigma_{uw}(T))^{C},$ then $\sigma_{uw}(T)=\sigma_{w}(T).$ From Lemma \ref{lem2},  $\Delta_{a}^g(T)=p_{0}(T)$ and $\sigma_{u bw}(T)=\sigma_{bw}(T).$
  \end{proof}

\section{On the extension  of  property $(b)$ }

We begin this section by the following definition, in which we introduce  new spectral properties named $(bz_{1})$ an $(gbz_{1})$  which are extensions (see Theorem \ref{thm1} bellow)  of  properties $(b)$ and $(gb),$ respectively.

\begin{dfn}\label{defn1} A bounded linear operator  $T\in L(X)$ is said to satisfy:\\
(i) Property $(bz_{1})$ if $\Delta_{uf}(T)= p_{00}(T),$ or equivalently
$\sigma_{a}(T)= \sigma_{uf}(T)\bigsqcup p_{00}(T).$   \\
(ii)  Property $(gbz_{1})$ if $\Delta_{uf}^g(T)= p_{0}(T),$ or equivalently  $\sigma_{a}(T)= \sigma_{ubf}(T)\bigsqcup p_{0}(T).$
\end{dfn}

\begin{ex}\label{ex1} Here and elsewhere the operators
R and L are
 defined on the Hilbert space $ l^2(\mathbb{N})$ (which is usually denoted by $\l^2$) by
$$ R(x_1,x_2,x_3,\ldots)=(0,x_1,x_2,x_3,\ldots) \mbox{ and }   L(x_1,x_2,x_3,\ldots)=(x_2,x_3,x_4,\ldots).$$

 \noindent 1.  It is well known that $\sigma_a(R)=C(0,1),$ where $C(0, 1)$ is the unit circle
of $ \mathbb{C}$ and so  $p_{00}(R)= p_{0}(R)=\emptyset.$ From \cite[Example 2.1]{kaoutar-zariouh}, we have
  $\sigma_{ubf}(R)=\sigma_{uf}(R)=C(0, 1).$ On the Banach space   $ l^2 \oplus \l^2,$ we define the operator $T$ by $T=R\oplus 0.$ Then $\sigma_{a}(T)=\sigma_{uf}(T)=C(0, 1)\cup\{0\}$ and $p_{00}(T)= p_{0}(T)=\emptyset.$ So $\Delta_{uf}(T)=p_{00}(T)$ and this means that $T$ satisfies property $(bz_{1}).$ On the other hand, it is easily seen that $\sigma_{ubf}(T)=C(0, 1).$ So $\Delta_{uf}^g(T)=\{0\}\neq p_{0}(T),$ and this means that     $T$  does not  satisfy property $(gbz_{1}).$ \\
  2. Let $T$ be the operator defined on $l^2$ by $T (x_1,x_2,x_3,\ldots)=(\frac{x_2}{2},\frac{x_3}{3},\frac{x_4}{4},\ldots).$ It is easily seen that
    $\sigma_a(T)=\sigma_{uf}(T)=\sigma_{ubf}(T)=\{0\}.$ Since $\mbox{asc}(A)=\infty,$ it follows that  $p_{0}(T)=p_{00}(T)=\emptyset.$ So  $\Delta_{uf}(T)= p_{00}(T)$ and  $\Delta_{uf}^g(T)= p_{0}(T);$ this means that  $T$ satisfies    properties $(gbz_{1})$ and $(bz_{1}).$\\
 3. Let $T$ be the operator defined on  the Banach space $l^2\oplus l^2$ by $T=L\oplus R.$ We have $\sigma_a(T)=\sigma_a(L)\cup\sigma_a(R)=D(0,1),$ where $D(0,1)$ is the closed unit disc of $ \mathbb{C}$   and so $ p_{00}(T)=p_{0}(T)=\emptyset.$ It is (see also \cite{kaoutar-zariouh}) easily seen that $ \sigma_{uf}(T)= C(0, 1).$ So $T$ does not satisfy neither property $(bz_{1})$ nor property $(gbz_{1}).$
\end{ex}

 As we have showed  in Example \ref{ex1}, there exist in nature operators who do not satisfy neither  property $(bz_{1})$  nor property $(gbz_{1}).$ Nonetheless,   we explore some conditions   which guarantee these properties for a bounded linear operator. We begin by  giving,   a relationship between properties   $(bz_{1})$ and $(b),$ and  between properties   $(gbz_{1})$ and $(gb).$

\begin{thm}\label{thm1}  The following statements hold for every  $T\in L(X).$ \\
(i) $T$ satisfies property $(bz_{1})$ if and only if $T$ satisfies property  $(b)$ and  $\sigma_{uf}(T)=\sigma_{uw}(T).$\\
(ii)  $T$ satisfies property $(gbz_{1})$ if and only if $T$ satisfies property  $(gb)$ and $\sigma_{ubf}(T)=\sigma_{ubw}(T).$
\end{thm}
\begin{proof} (i) Suppose that satisfies property $(bz_{1}),$ that's $\Delta_{uf}(T)=p_{00}(T).$ As $\Delta_{a}(T)\subset\Delta_{uf}(T),$ then   $\Delta_{a}(T)\subset p_{00}(T)$ and since the inclusion $\Delta_{a}(T)\supset p_{00}(T)$ is always  true, it follows that $\Delta_{a}(T)= p_{00}(T).$ Moreover, we have $\sigma_{uf}(T)=\sigma_a(T)\setminus p_{00}(T)=\sigma_{uw}(T).$ The converse is obvious.\\
(ii) Assume that satisfies property $(gbz_{1}),$ that's $\Delta_{uf}^g(T)=p_{0}(T).$ As $\Delta_{a}^g(T)\subset\Delta_{uf}^g(T),$ then   $\Delta_{a}^g(T)\subset p_{0}(T)$ and since the inclusion $\Delta_{a}^g(T)\supset p_{0}(T)$ is always  true, it follows that $\Delta_{a}^g(T)= p_{0}(T).$  Furthermore, we have $\sigma_{ubf}(T)=\sigma_a(T)\setminus p_{0}(T)=\sigma_{ubw}(T).$ The converse is also clear.
\end{proof}

\begin{rema} Generally, property $(b)$  and property $(gb)$  do not imply property $(bz_{1})$ and property  $(gbz_{1}),$  respectively. For this,  we consider the left shift operator $L$ defined on $l^2,$ we have $\sigma_a(L)=\sigma_{ubw}(L)=\sigma_{uw}(L)=D(0,1)$ and $p_{0}(L)=\emptyset.$ So $L$ satisfies property $(gb)$ and then property $(b).$ It is already showed   in \cite[Example 2.3]{kaoutar-zariouh} that   $\sigma_{uf}(L)=C(0,1).$     Thus   $\Delta_{uf}(L)\not=\emptyset=p_{00}(L)$ and  $L$ does not satisfy  property $(bz_{1}).$ Hence it does not satisfy  property $(gbz_{1}).$ Note also that from \cite[Remark 6]{kaoutar-zariouh}, we have  $\sigma_{ubf}(L)=C(0,1).$
\end{rema}

 In the next corollary, we show that an operator satisfying property $(gbz_{1}),$ satisfies property $(bz_{1}).$ The operator $T=R\oplus 0$ given in    Example \ref{ex1}, proves that the converse is not generally  true. Furthermore, we give a condition of equivalence between them.

\begin{cor}\label{cor1} Let $T\in L(X).$ The following statements are equivalent.\\
(i)  $T$ satisfies property $(gbz_{1});$ \\
(ii)  $T$ satisfies property $(bz_{1})$ and $p_{0}(T)=p_{0}^a(T).$
\end{cor}

\begin{proof} (i) $\Longrightarrow$ (ii) The inclusion  $\Delta_{uf}(T)\subset\Delta_{uf}^g(T)$ is always true. As $T$ satisfies property $(gbz_{1})$ then $\Delta_{uf}(T)\subset p_{0}(T).$ Hence $\Delta_{uf}(T)= p_{00}(T).$ Moreover,

 $  p_{0}(T)=\mbox{iso}\,\sigma_a(T)\cap(\sigma_{ubf}(T))^C=\mbox{iso}\,\sigma_a(T)\cap(\sigma_{ubw}(T))^C = p_{0}^a(T).$\\
(i) $\Longleftarrow$ (ii) Is an immediate consequence of  Theorems \ref{thm1} and \ref{thm0}.
\end{proof}

\begin{cor}\label{cor2} For every $T\in L(X),$ we have the following statements.\\
(i) $T$ satisfies property $(bz_{1})$ if and only if $T$ satisfies $(bz)$ and $p_{00}(T)=p_{00}^a(T).$\\
(ii) $T$ satisfies property $(gbz_{1})$ if and only if $T$ satisfies $(gbz)$ and $p_{0}(T)=p_{0}^a(T).$
\end{cor}

\begin{proof} (i) Since $T$ satisfies property $(bz_{1}),$ then

  $ p_{00}(T)=\mbox{iso}\,\sigma_a(T)\cap(\sigma_{uf}(T))^C=\mbox{iso}\,\sigma_a(T)\cap(\sigma_{uw}(T))^C = p_{00}^a(T).$ So  $T$ satisfies property $(bz).$ The converse is clear.\\
  (ii) Follows directly from  Corollary \ref{cor1},  the first assertion (i) and Theorem \ref{thm0}.
\end{proof}

\begin{rema} From Corollary \ref{cor2}, if $T\in L(X)$ satisfies property  $(bz_{1})$ or property  $(gbz_{1}),$ then it satisfies property $(bz).$ But the converse is not true, as  the  following example shows.  We consider the operator $T$ defined on the Banach space  $l^2\oplus l^2$ by $T=R\oplus P,$ where  $P$ is the operator defined on $l^2$ by $ P(x_1,x_2,x_3,\ldots)=(0,x_2,x_3,x_4,\ldots).$ Then  $\sigma_a(T)=C(0, 1)\cup\{0\},$ $\sigma_{uf}(T)=\sigma_{uw}(T)=C(0, 1),$ $p_{0}^a(T)=p_{00}^a(T)=\{0\}.$
So $\Delta_{uf}(T)=p_{00}^a(T).$ This means that $T$ satisfies property  $(bz)$ [or equivalently $(gbz)$]. But $T$ does not satisfy neither property $(bz_{1})$ nor  property $(gbz_{1}),$ since $\sigma_{ubf}(T)=\sigma_{ubw}(T)=C(0, 1)$ and $p_{0}(T)=p_{00}(T)=\emptyset.$
\end{rema}

\begin{cor}\label{cor3} Let   $T\in L(X).$ If $T$ and $T^*$ have the  SVEP on $(\sigma_{uf}(T))^C,$ then $T$ satisfies properties $(gbz_{1})$ and  $(bz_{1}).$
\end{cor}

\begin{proof} As  $T$ have the  SVEP on $(\sigma_{uf}(T))^C$ then from  \cite[Theorem 2.7]{kaoutar-zariouh}, $T$ satisfies property $(gbz).$ By Lemma \ref{lem2}, the SVEP for $T^*$ on  $(\sigma_{uf}(T))^C$ implies  that $p_{0}(T)=p_{0}^a(T).$ Thus by Corollary \ref{cor2}, property $(gbz_{1})$ holds for $T$ and then property $(bz_{1})$ holds also for $T.$
\end{proof}

\begin{rema} We cannot guarantee the property $(gbz_{1}),$ for an operator $T,$  if we restrict  only to the SVEP of $T$  on $(\sigma_{uf}(T))^C$ or  to the SVEP of $T^*$ on $(\sigma_{uf}(T))^C.$ The operator $T=R\oplus P$ defined above has the SVEP, but it does not satisfy properties $(gbz_{1})$ and $(bz_{1}).$ Here  $T$ has the SVEP, but ${\mathcal S}(T^*)={\mathcal S}(L)=\{ \lambda\in\mathbb{C}: 0\leq |\lambda|< 1 \}$  and $\sigma_{uf}(T)=C(0, 1).$ On the other hand, It already mentioned that the left shift operator $L$ does  not satisfy the properties $(bz_{1})$ and $(gbz_{1}),$ even if $L^*=R$ has the SVEP.
\end{rema}

\section{On the extended  Rako\v{c}evi\'{c}'s property}

Similarly to the Definition \ref{defn1}, we introduce in the next definition a new extension of Rako\v{c}evi\'{c}'s property $(w),$  and a new extension of property $(gw).$

\begin{dfn}\label{defn2}
 A bounded linear operator  $T\in L(X)$ is said to satisfy:\\
  i) Property $(w_{\pi_{00}})$ if $\Delta_{uf}(T)= \pi_{00}(T),$ or equivalently
$\sigma_{a}(T)= \sigma_{uf}(T)\bigsqcup \pi_{00}(T).$ \\
  ii)  Property $(gw_{\pi_{00}})$ if $\Delta_{uf}^g(T)= \pi_{0}(T),$ or equivalently
$\sigma_{a}(T)= \sigma_{ubf}(T)\bigsqcup \pi_{0}(T).$
\end{dfn}

\begin{ex} \label{ex2} ~~

 \noindent1. The operator $T=R\oplus 0$ given in    Example \ref{ex1}, satisfies property $(w_{\pi_{00}}),$ since $\Delta_{uf}(T)=\pi_{00}(T)=\emptyset.$ But it does not satisfy property  $(gw_{\pi_{00}}),$ since $\Delta_{uf}^g(T)=\{0\}\neq \pi_{0}(T)=\emptyset.$ Observe that $\sigma(T)=D(0, 1).$  \\
2. Let $I$ the identity operator on an infinite dimensional Banach space $X.$ It is clear that $\sigma_{a}(I)=\sigma_{uf}(I)=\pi_{0}(I)=\{1\}$ and  $\sigma_{ubf}(I)=\pi_{00}(I)=\emptyset.$ So $\Delta_{uf}(I)=\emptyset=\pi_{00}(I)$ and $\Delta_{uf}^{g}(I)=\{1\}=\pi_{0}(I).$ So the identity operator satisfies the  properties $(w_{\pi_{00}})$ and $(gw_{\pi_{00}}).$\\
3. Let $L$ be the left shift operator defined on $l^2.$ It is already mentioned above  that   $\sigma_{uf}(L)=\sigma_{ubf}(L)=C(0,1).$ As $\sigma_a(L)=\sigma(L)=D(0, 1),$ then  $\pi_{0}(L)=\pi_{00}(L)=\emptyset.$ So the left shift operator  does not satisfy neither  property $(w_{\pi_{00}})$ nor  property $(gw_{\pi_{00}}).$
\end{ex}

We prove in the next theorem that properties    $(w_{\pi_{00}})$ and $(gw_{\pi_{00}})$ are respectively, extensions of properties  $(w)$ and   $(gw).$

\begin{thm}\label{thm2}  The following statements hold for every  $T\in L(X).$ \\
(i) $T$ satisfies property $(w_{\pi_{00}})$ if and only if $T$ satisfies property  $(w)$ and  $\sigma_{uf}(T)=\sigma_{uw}(T).$\\
(ii)  $T$ satisfies property $(gw_{\pi_{00}})$ if and only if $T$ satisfies property  $(gw)$ and $\sigma_{ubf}(T)=\sigma_{ubw}(T).$
\end{thm}

\begin{proof}(i) Suppose that  $T$ satisfies property $(w_{\pi_{00}}),$ that's $\Delta_{uf}(T)= \pi_{00}(T).$   Then $\Delta_{a}(T)\subset\pi_{00}(T).$  If  $\lambda\in\pi_{00}(T),$ then $\lambda\in \mbox{iso}\,\sigma(T)$  and so   $T$ has the SVEP at $\lambda.$ Since $T-\lambda I$ is upper semi-Fredholm, then $\mbox{asc}(T-\lambda I)$ is finite and then   $\lambda\not\in\sigma_{uw}(T).$  Hence $\Delta_{a}(T)=\pi_{00}(T),$   and  consequently,  $\sigma_{uf}(T)=\sigma_a(T)\setminus\pi_{00}(T)=\sigma_{uw}(T).$   The converse is clear.\\
(ii)   $\Delta_{a}^g(T)\subset\pi_{0}(T),$ since $T$ satisfies property $(gw_{\pi_{00}}).$ Moreover,  $\pi_{0}(T)=\mbox{iso}\,\sigma_a(T)\cap(\sigma_{ubf}(T))^C.$  By the proof of
\cite[Theorem 2.8]{Berkani-koliha}, $\pi_{0}(T)=\mbox{iso}\,\sigma_a(T)\cap(\sigma_{ubw}(T))^C\subset\Delta_{a}^g(T).$  Hence $\Delta_{a}^g(T)=\pi_{0}(T),$ and this means that  $T$ satisfies property  $(gw)$. Moreover,  $\sigma_{ubf}(T)=\sigma_a(T)\setminus\pi_{0}(T)=\sigma_{ubw}(T).$
 The converse is also clear.
\end{proof}

\begin{rema} Note that the sufficient condition ``$\sigma_{uf}(T)=\sigma_{uw}(T)$" [which is equivalent to the condition ``$\sigma_{ubf}(T)=\sigma_{ubw}(T)$"] assumed in the two statements   of Theorem \ref{thm2} is crucial, as we can see in the following example: The left shift operator $L$ satisfies  properties   $(gw)$ and  $(w),$ since $\Delta_{a}^g(T)=\pi_{0}(L)=\emptyset$ and $\Delta_{a}(L)=\pi_{00}(L)=\emptyset.$ But it does not satisfy neither $(w_{\pi_{00}})$ nor $(gw_{\pi_{00}}),$ since
$\Delta_{uf}(L)\neq\emptyset=\pi_{00}(L)$ and $\Delta_{uf}^g(L)\neq\emptyset=\pi_{0}(L).$
\end{rema}

In the next corollary, we show that an operator satisfying property $(gw_{\pi_{00}}),$ satisfies property $(w_{\pi_{00}}).$
The operator $T=S\oplus 0$ defined on $ l^2\oplus l^2,$  where $S$ is defined on $l^2$ by $ S(x_{1},x_{2},...)=(0,\frac{x_{1}}{2},\frac{x_{2}}{3},\ldots)$ proves that the converse is not generally  true. Indeed,   we have $\sigma_{a}(T)=\sigma_{uf}(T)=\{0\}$ and $\pi_{00}(T)=\emptyset.$ So $T$ satisfies property  $(w_{\pi_{00}}).$ But it does not satisfy property $(gw_{\pi_{00}}),$ since  $\pi_{0}(T)=\{0\}.$

 Furthermore, we give a condition of equivalence between them.

\begin{cor}\label{cor3}Let $T\in L(X).$ The following statements are equivalent.\\
(i)  $T$ satisfies property $(gw_{\pi_{00}});$ \\
(ii)  $T$ satisfies property $(w_{\pi_{00}})$ and $\pi_{0}(T)=p_{0}^a(T).$
\end{cor}

\begin{proof}(i) $\Longrightarrow$ (ii)  Since   $\Delta_{uf}(T)\subset\Delta_{uf}^g(T)$ and  $T$ satisfies property $(gw_{\pi_{00}}),$ then $\Delta_{uf}(T)\subset \pi_{00}(T).$ On the other hand, $\pi_{00}(T)\subset\pi_{0}(T)=\Delta_{uf}^g(T).$ Hence $\Delta_{uf}(T)= \pi_{00}(T).$ Moreover,

 $ \pi_{0}(T)=\mbox{iso}\,\sigma_a(T)\cap(\sigma_{ubf}(T))^C=\mbox{iso}\,\sigma_a(T)\cap(\sigma_{ubw}(T))^C = p_{0}^a(T).$\\
 (i) $\Longleftarrow$ (ii) Is an immediate consequence of  Theorems \ref{thm2} and \ref{thm0}.
\end{proof}

\begin{rema} \label{rema} It is easy to get  that  if $T\in L(X)$ satisfies property $(w_{\pi_{00}}),$ then $$p_{00}(T)=\pi_{00}(T)=p_{00}^a(T).$$ In addition, if $T$ satisfies property $(gw_{\pi_{00}}),$ then $p_{0}(T)=\pi_{0}(T)=p_{0}^a(T),$ and this implies that   $p_{00}(T)=\pi_{00}(T)=p_{00}^a(T).$
\end{rema}

According to \cite{kaoutar-zariouh}, an operator $T\in L(X)$ is said to satisfy property $(w_{\pi_{00}^{a}})$ if $\Delta_{uf}(T)=\pi_{00}^{a}(T)$ and is said to satisfy property  $(gw_{\pi_{00}^{a}})$ if $\Delta_{uf}^g(T)=\pi_{0}^{a}(T).$

\begin{cor}\label{cor4} For every $T\in L(X),$ we have the following statements.\\
(i) $T$ satisfies property $(w_{\pi_{00}})$ if and only if $T$ satisfies $(bz_{1})$ and $\pi_{00}(T)=p_{00}(T).$\\
(ii)  $T$ satisfies property $(gw_{\pi_{00}})$ if and only if $T$ satisfies $(gbz_{1})$ and $\pi_{0}(T)=p_{0}(T).$
\end{cor}

\begin{rema}   There  exist operators satisfying property $(bz_{1})$ [resp., property $(gbz_{1})$]  which do not satisfy property $(w_{\pi_{00}})$  [resp., property $(gw_{\pi_{00}})$].

\noindent1.     Let $T \in L(l^2)$ be  defined by $T(x_{1}, x_{2},...)=( \frac{x_{2}}{2},\frac{x_{2}}{3},\ldots).$  Property $(bz_{1})$  holds for $T,$ since  $\sigma_{a}(T)=\sigma_{uf}(T)=\{0\}$ and $p_{00}(T)=\emptyset.$ While property $(w_{\pi_{00}})$ does not hold for $T,$ since  $\pi_{00}(T)=\{0\}.$ Note that here $p_{00}(T)\neq\pi_{00}(T).$

\noindent2. Let $Q$ be defined on $l^1(\mathbb{N})$ by $Q(x_1,x_2,...)=(0,\alpha_{1}x_1, \alpha_{2}x_2,\ldots, a_kx_k, \ldots),$ where $(\alpha_{i})$ is a sequence of complex numbers such that $0<|\alpha_{i}|\leq 1$ and $\sum_{i=1}^{\infty} |\alpha_{i}|<\infty. $\\
And we define the operator  $T$ on $l^1(\mathbb{N})\oplus l^1(\mathbb{N})$ by $T=Q\oplus 0.$ We have $\sigma(T)=\sigma_{a}(T)=\{0\}$ and $\pi_{0}(T)=\{0\}.$ It is easily seen that   $\R(T^n)$ is not closed for any $n\in\mathbb{N},$ so that   $\sigma_{ubf}(T)=\{0\}$ and $p_{0}(T)=\emptyset.$ Thus  $\Delta_{uf}^{g}(T)=p_{0}(T)$ and $\Delta_{uf}^{g}(T)\neq\pi_{0}(T),$ as desired. Note also that here $p_{0}(T)\neq \pi_{0}(T).$
\end{rema}

\begin{rema} ~~

\noindent1.  The properties  $(w_{\pi_{00}})$ and $(w_{\pi_{00}^{a}})$ are independent.\\
We consider   the operator $S=R\oplus 0$ defined  on the Banach space $l^2\oplus \mathbb{C}^n.$ We have  $\sigma_{a}(S)=C(0, 1)\cup\{0\},$   $\sigma_{uf}(S)=C(0, 1),$ $\pi_{00}^a(S)=\{0\}$ and  $\pi_{00}(S)=\emptyset.$ So $\Delta_{uf}(S)=\pi_{00}^a(S),$ but $\Delta_{uf}(S)\neq\pi_{00}(S).$ \\
Now,  we consider the operator $U=R\oplus T$ defined on the $l^2\oplus l^2,$ where $T$ is defined  $l^2$ by $T(x_{1}, x_{2},...)=( \frac{x_{2}}{2},\frac{x_{2}}{3},\ldots).$ Then  $\sigma_{a}(U)=C(0, 1)\cup\{0\},$   $\sigma_{uf}(U)=C(0, 1)\cup\{0\},$ $\pi_{0}^a(U)=\pi_{00}^a(U)=\{0\}$ and  $\pi_{0}(U)=\pi_{00}(U)=\emptyset.$ So $\Delta_{uf}(U)=\pi_{00}(U),$ but $\Delta_{uf}(U)\neq\pi_{00}^a(U).$

\noindent2.  The properties  $(gw_{\pi_{00}})$ and $(gw_{\pi_{00}^{a}})$ are independent.\\
   We take the operator $V=R\oplus A$ defined $l^2\oplus l^2,$ where $A$ is defined by $A(x_{1},x_{2},\ldots)=(0,\frac{-x_{1}}{2},0,\ldots).$ $V$ satisfies property  $(gw_{\pi_{00}^{a}}),$ since $\sigma_{a}(V)=C(0, 1)\cup\{0\},$ $\sigma_{ubf}(V)=C(0, 1)$ and  $\pi_{0}^{a}(V)=\{0\}.$ But, it does not satisfy property $(gw_{\pi_{00}}),$ since $\pi_{0}(V)=\emptyset.$ On the other hand, the operator $U$ considered in the first point, does not satisfy property $(w_{\pi_{00}^{a}}),$ and this implies  that it does not  satisfy also  property $(gw_{\pi_{00}^{a}}).$ But, it satisfies property $(gw_{\pi_{00}}),$ since $\sigma_{ubf}(U)=C(0, 1)\cup\{0\}.$

\end{rema}

\noindent {\bf \underline{Conclusion}:}

As  conclusion, we  give a summary of the results obtained in
this paper. In the following diagram, which extends the one that has been   presented  in \cite{kaoutar-zariouh},   arrows signify implications
between the properties introduced in this paper and those that have been introduced in \cite{kaoutar-zariouh},  and other Weyl type
theorems (generalized or not). The numbers near the arrows are references to the results
in the present paper (numbers without brackets) or to the
bibliography therein (numbers in square brackets).

\begin{center}
 \vspace{5pt} \small{
\vbox{
\[
\begin{CD}
gaW@<\mbox{{\scriptsize\cite{kaoutar-zariouh}}}<<({gw}_{\pi_{00}^{a}})
@>\mbox{{\scriptsize\ref{thm2}}}>>(w_{\pi_{00}^{a}})@>\mbox{{\scriptsize\cite{kaoutar-zariouh}}}>>aW@<\mbox{{\scriptsize\cite{Berkani-koliha}}}<<gaW\\
@VV\mbox{{\scriptsize\cite{Berkani-koliha}}}V@VV\mbox{{\scriptsize\cite{kaoutar-zariouh}}}V@VV\mbox{{\scriptsize\cite{kaoutar-zariouh}}}V
@VV\mbox{{\scriptsize\cite{rakocevic1}}}V
@VV\mbox{{\scriptsize\cite{Berkani-koliha}}}V\\
gaB@<<\mbox{{\scriptsize\cite{kaoutar-zariouh}}}<(gbz)&  \,\,\,\Longleftrightarrow_{\mbox{{\scriptsize\cite{kaoutar-zariouh}}}}&(bz)@>>\mbox{{\scriptsize\cite{kaoutar-zariouh}}}>aB
&  \,\,\,\Longleftrightarrow_{\mbox{{\scriptsize\cite{amuch-zguitti}}}} &gaB\\
@.@AA\mbox{{\scriptsize\ref{cor2}}}A@AA\mbox{{\scriptsize\ref{cor2}}}A\\
(gw_{\pi_{00}})@>>\mbox{{\scriptsize\ref{cor4}}}>(gbz_{1})@>>\mbox{{\scriptsize\ref{cor1}}}>(bz_{1})
@<<\mbox{{\scriptsize\ref{cor4}}}<(w_{\pi_{00}})\\
@VV\mbox{{\scriptsize\ref{thm2}}}V@VV\mbox{{\scriptsize\ref{thm1}}}V@VV\mbox{{\scriptsize\ref{thm1}}}V@VV\mbox{{\scriptsize\ref{thm2}}}V\\
(gw)@>>\mbox{{\scriptsize\cite{berkani-zariouh1}}}>(gb)@>>\mbox{{\scriptsize\cite{berkani-zariouh1}}}>(b) @<<\mbox{{\scriptsize\cite{berkani-zariouh1}}}<(w)
\end{CD}
\]}}
\end{center}

Moreover, counterexamples were given to show that the reverse of each implication (presented with number without bracket) in
the diagram is not true. Nonetheless, it was proved that under some extra assumptions, these
implications are equivalences.\\

\noindent {\bf \underline{Remark }:}  This paper will be followed by a second one (in preparation), in which we will consider a ``Weyl-type''  version of the
results obtained, by introducing and studying the following properties.

\noindent {\bf \underline{Definition}:}[article in reparation] An operator $T\in L(X)$ is said to satisfy:\\
property  $(bz_{2})$ if $\Delta_{uf}(T)=p_{0}(T).$ \\
property  $(bz_{3})$ if $\Delta_{uf}(T)=p_{0}^{a}(T).$\\
property  $(w_{\pi_{0}})$ if $\Delta_{uf}(T)=\pi_{0}(T).$\\
property  $(w_{\pi_{0}^{a}})$ if $\Delta_{uf}(T)=\pi_{0}^{a}(T).$\\

\goodbreak
 {\small \noindent Kaoutar Ben Ouidren,\newline Laboratory (L.A.N.O), Department of
Mathematics,\newline Faculty of Science, Mohammed I University,\\
 \noindent Oujda 60000 Morocco.\\
 \noindent benouidrenkaoutar@gmail.com\\

  \noindent Hassan  Zariouh,\newline Department of
Mathematics (CRMEFO),\newline
 \noindent and laboratory (L.A.N.O), Faculty of Science,\newline
  Mohammed I University, Oujda 60000 Morocco.\\
 \noindent h.zariouh@yahoo.fr

\end{document}